\documentclass[11pt]{amsart}

\usepackage{hyperref}

\usepackage{mathrsfs}
\usepackage{bbm}
\usepackage{cases}
\usepackage{amssymb}

\usepackage{amsmath,latexsym,amssymb,amsthm,array,amsfonts}
\usepackage{mathrsfs,dsfont,bbm}

\newtheorem{theorem}{Theorem}[section]
\newtheorem{lemma}{Lemma}[section]
\newtheorem{corollary}{Corollary}[section]

\newtheorem{remark}{Remark}
\newtheorem{definition}{Definition}[section]
\newtheorem*{exampless}{Theorem A}

\numberwithin{equation}{section}

\begin{document}

\title[Generalized quadratic covariation]
{The generalized quadratic covariation for fractional Brownian
motion with Hurst index less than $1/2$ $^{*}$}

\footnote[0]{$^{*}$The Project-sponsored by NSFC (10571025).}

\author[L. Yan, C. Chen and J. Liu]
{Litan Yan${}^{\dag,\S}$, Chao Chen${}^{\ddag}$ and Junfeng
Liu${}^{\natural}$}

\footnote[0]{${}^{\S}$Corresponding author (litanyan@hotmail.com).}

\keywords{fractional Brownian motion,  Malliavin calculus, local
time, fractional It\^{o} formula, quadratic covariation.}

\subjclass[2000]{Primary 60G15, 60H05; Secondary 60H07}

\maketitle

\date{}

\begin{center}
{\it ${}^{\dag}$Department of Mathematics, Donghua University\\
2999 North Renmin Rd., Songjiang, Shanghai 201620, P.R. China}\\
{\it ${}^{\ddag}$Department of Mathematics, East China University of
Science and Technology\\
130 Mei Long Rd., Xuhui, Shanghai 200237, P.R. China}\\
{\it ${}^{\natural}$Department of Mathematics, Nanjing Audit
University, 86 West Yushang Rd., Nanjing 211815, P.R. China}
\end{center}


\maketitle

\begin{abstract}
Let $B^H$ be a fractional Brownian motion with Hurst index
$0<H<1/2$. In this paper we study the {\it generalized quadratic
covariation} $[f(B^H),B^H]^{(W)}$ defined by
$$
[f(B^H),B^H]^{(W)}_t=\lim_{\varepsilon\downarrow
0}\frac{2H}{\varepsilon^{2H}}\int_0^t\left\{f(B^{H}_{
s+\varepsilon})
-f(B^{H}_s)\right\}(B^{H}_{s+\varepsilon}-B^{H}_s)s^{2H-1}ds,
$$
where the limit is uniform in probability and $x\mapsto f(x)$ is a
deterministic function. We construct a Banach space ${\mathscr H}$
of measurable functions such that the generalized quadratic
covariation exists in $L^2$ and the Bouleau-Yor identity takes the
form
$$
[f(B^H),B^H]_t^{(W)}=-\int_{\mathbb {R}}f(x){\mathscr L}^{H}(dx,t)
$$
provided $f\in {\mathscr H}$, where ${\mathscr L}^{H}(x,t)$ is the
weighted local time of $B^H$. This allows us to write the fractional
It\^{o} formula for absolutely continuous functions with derivative
belonging to ${\mathscr H}$. These are also extended to the
time-dependent case.
\end{abstract}

\section{Introduction}
Given $H\in (0,1)$, a fractional Brownian motion (fBm) with Hurst
index $H$ is a mean zero Gaussian process $B^H=\{B_t^H, 0\leq t\leq
T\}$ such that
$$
E\left[B_t^HB_s^H\right]=\frac{1}{2}\left[t^{2H}+s^{2H}-|t-s|^{2H}
\right]
$$
for all $t,s\geqslant 0.$ For $H=1/2$, $B^H$ coincides with the
standard Brownian motion $B$. $B^H$ is neither a semimartingale nor
a Markov process unless $H=1/2$, so many of the powerful techniques
from stochastic analysis are not available when dealing with $B^H$.
As a Gaussian process, one can construct the stochastic calculus of
variations with respect to $B^{H}$. Some surveys and complete
literatures for fBm could be found in Biagini {\it et
al}~\cite{BHOZ}, Decreusefond and \"Ust\"unel~\cite{Dec}, Gradinaru
{\it et al}~\cite{Grad1,Grad2}, Hu~\cite{Hu2},
Mishura~\cite{Mishura2} and Nualart~\cite{Nua4}. It is well-known
that the usual quadratic variation $[B^{H},B^{H}]_t=0$ for $2H>1$
and $[B^{H},B^{H}]_t=\infty$ for $2H<1$, where
$$
\left[B^{H},B^{H}\right]_t=\lim_{\varepsilon\downarrow
0}\frac{1}{\varepsilon}\int_0^t(B^H_{s+\varepsilon}-B^H_s)^2ds
$$
in probability. Clearly, we have also
$$
[B^{H},B^{H}]_t=\lim_{n\to
\infty}\sum_{j=1}^n\left(B^H_{jt/n}-B^H_{(j-1)t/n}\right)^2,
$$
where the limit is uniform in probability. This is inconvenience to
some studies and applications for fBm. We need to find a
substitution tool. Recently, Gradinaru {\it et al}~\cite{Grad1} (see
also~\cite{Grad2} and the references therein) have introduced some
substitution tools and studied some fine problems. They introduced
firstly an It\^o formula with respect to a symmetric-Stratonovich
integral, which is closer to the spirit of Riemann sums limits, and
defined a class of high order integrals having an interest by
themselves. On the other hand, inspired by
Gradinaru-Nourdin~\cite{Grad3,Grad} and Nourdin {\it et
al}~\cite{Nour1,Nour2}, as the substitution tool of the quadratic
variation, Yan {\it et al}~\cite{Yan1} considered the {\it
generalized quadratic covariation}, and proved its existence for
$\frac12<H<1$ (Thanks to the suggestions of some Scholars we use the
present appellation).
\begin{definition}
Let $0<H<1$ and let $f$ be a measurable function on ${\mathbb R}$.
The limit
\begin{equation}\label{sec1-eq1.1}
\lim_{\varepsilon\downarrow
0}\frac{2H}{\varepsilon^{2H}}\int_0^t\left\{f(B^{H}_{
s+\varepsilon})
-f(B^{H}_s)\right\}(B^{H}_{s+\varepsilon}-B^{H}_s)s^{2H-1}ds
\end{equation}
is called the {\it generalized quadratic covariation} of $f(B^H)$
and $B^H$, denoted by $[f(B^H),B^H]^{(W)}_t$, provided the limit
exists uniformly in probability.
\end{definition}
In particular, we have
$$
[B^H,B^H]^{(W)}_t=t^{2H}
$$
for all $0<H<1$. If $H=\frac12$, the generalized quadratic
covariation coincides with the usual quadratic covariation of
Brownian motion $B$. For $\frac12<H<1$, Yan {\it et al}~\cite{Yan3}
showed the generalized quadratic covariation can also be defined as
\begin{equation}\label{sec1-eq1.2}
\left[f(B^H),B^H\right]^{(W)}_t=2H\lim_{\|\pi_n\|\to 0}\sum_{t_j\in
\pi_n}\left(\Lambda_j\right)^{2H-1}
\{f(B^H_{t_j})-f(B^H_{t_{j-1}})\} (B^H_{t_j}-B^H_{t_{j-1}}),
\end{equation}
provided the limit exists
uniformly in probability, where $\pi_n=\{0=t_0<t_1 <\cdots< t_n=t\}$
denotes an arbitrary partition of the interval $[0, t]$ with
$\|\pi_n\|= \sup_{j}(t_{j}-t_{j-1})\to 0$, and
$\Lambda_{j}=\frac{t_j}{t_j-t_{j-1}}$, $j=1,2,\ldots,n$.
Moreover, by applying the time reversal $\widehat{B}^H_t=B^H_{T-t}$
on $[0,T]$ and the integral
$$
\int_{\mathbb R}f(x)\mathscr{L}^{H}(dx,t),
$$
Yan {\it et al}~\cite{Yan3} constructed a Banach space ${\mathbb
B}_H$ of measurable functions such that the generalized quadratic
covariation $[f(B^H),B^H]^{(W)}_t$ exists in $L^2$ if $f\in {\mathbb
B}_H$, where
$$
{\mathscr L}^{H}(x,t)=2H\int_0^t\delta(B^{H}_s-x)s^{2H-1}ds
$$
is the weighted local time of fBm $B^{H}$. However, when
$0<H<\frac12$ the method used in Yan {\it et al}~\cite{Yan1,Yan3} is
inefficacy. In the present paper, we shall consider the generalized
quadratic covariation with $0<H<\frac12$. Our start point is to
consider the decomposition
\begin{equation}\label{sec1-eq1.3}
\begin{split}
&\frac{1}{\varepsilon^{2H}}\int_0^t\left\{f(B^{H}_{
s+\varepsilon})-f(B^{H}_s)\right\}(B^{H}_{s+\varepsilon}-B^{H}_s)ds^{2H}\\
&=\frac{1}{\varepsilon^{2H}}\int_0^tf(B^{H}_{
s+\varepsilon})(B^{H}_{s+\varepsilon}-B^{H}_s)ds^{2H}-\frac{1}{\varepsilon^{2H}}
\int_0^tf(B^{H}_s)(B^{H}_{s+\varepsilon}-B^{H}_s)ds^{2H}.
\end{split}
\end{equation}
Clearly, if the modulus in expression~\eqref{sec1-eq1.3} is $\frac{1}{\varepsilon}$, the decomposition is meaningless in general. For example, for $f(x)=x$ we have
\begin{align*}
\frac{1}{\varepsilon}
\int_0^tE\left[B^{H}_s(B^{H}_{s+\varepsilon}-B^{H}_s)\right]ds^{2H}
&=\frac{1}{\varepsilon}
\int_0^t\frac12\left[(s+\varepsilon)^{2H}-s^{2H}-\varepsilon^{2H}\right]ds^{2H}
\\
&\longrightarrow -\infty,
\end{align*}
as $\varepsilon\downarrow 0$. However,
\begin{align*}
\frac{1}{\varepsilon^{2H}}
&\int_0^t\left|EB^{H}_s(B^{H}_{s+\varepsilon}-B^{H}_s)\right|ds^{2H}\\
&=\frac{1}{\varepsilon^{2H}}
\int_0^t\frac12\left[s^{2H}+\varepsilon^{2H}-(s+\varepsilon)^{2H}\right]ds^{2H}\longrightarrow \frac12t^{2H},
\end{align*}
as $\varepsilon\downarrow 0$. Thus, for $0<H<\frac12$ we can
consider the decomposition~\eqref{sec1-eq1.3}. By estimating the two
terms of the right hand side in the
decomposition~\eqref{sec1-eq1.3}, respectively, we can construct a
Banach space ${\mathscr H}$ of measurable functions $f$ on $\mathbb
R$ such that $\|f\|_{\mathscr H}<\infty$, where
\begin{align*}
\|f\|_{\mathscr H}=\sqrt{\int_0^T\int_{\mathbb
R}|f(x)|^2e^{-\frac{x^2}{2s^{2H}}}\frac{dxds}{\sqrt{2\pi}s^{1-H}}}
+\sqrt{\int_0^T\int_{\mathbb R}|f(x)|^2
e^{-\frac{x^2}{2s^{2H}}}\frac{dxds}{\sqrt{2\pi}(T-s)^{1-H}}}.
\end{align*}
We show that {\it generalized quadratic covariation}
$[f(B^H),B^H]_t^{(W)}$ exists in $L^2$ for all $t\in [0,T]$ if $f\in
{\mathscr H}$. This allows us to write It\^{o}'s formula for
absolutely continuous functions with derivative belonging to
${\mathscr H}$ and to give the Bouleau-Yor identity. It is important
to note that the decomposition~\eqref{sec1-eq1.3} is inefficacy for
$\frac12<H<1$.

This paper is organized as follows. In Section~\ref{sec2} we present
some preliminaries for fBm. In Section~\ref{sec3}, we establish some
technical estimates associated with fractional Brownian motion with
$0<H<\frac12$. In Section~\ref{sec4}, we prove the existence of the
generalized quadratic covariation. We construct the Banach space
${\mathscr H}$ such that the generalized quadratic covariation
$[f(B^H),B^H]^{(W)}$ exists in $L^2$ for $f\in {\mathscr H}$. As an
application we show that the It\^o type formula
(F\"ollmer-Protter-Shiryayev's formula)
$$
F(B^H)=F(0)+\int_0^tf(B^H_s)dB^H_s+\frac12\left[f(B^H),B^H
\right]^{(W)}_t
$$
holds, where $F$ is an absolutely continuous function with the
derivative $F'=f\in {\mathscr H}$. In Section~\ref{sec5}, we
introduce the integral of the form
\begin{equation}\label{sec1-eq1.5}
\int_{\mathbb {R}}f(x){\mathscr L}^{H}(dx,t),
\end{equation}
where $x\mapsto f(x)$ is a deterministic function. We show that the
integral~\eqref{sec1-eq1.5} exists in $L^2$, and the Bouleau-Yor
identity takes the form
$$
[f(B^H),B^H]_t^{(W)}=-\int_{\mathbb {R}}f(x){\mathscr L}^{H}(dx,t)
$$
provided $f\in {\mathscr H}$. Moreover, by applying the
integral~\eqref{sec1-eq1.5} we show that~\eqref{sec1-eq1.1}
and~\eqref{sec1-eq1.2} coincide for $0<H<\frac12$ when $f\in
{\mathscr H}$. In Section~\ref{sec6}, we consider the time-dependent
case, and define the local time of $B^H$ with $0<H<\frac12$ on a
continuous curve.


\section{Preliminaries}\label{sec2}
In this section, we briefly recall some basic definitions and
results of fBm. For more aspects on these material we refer to
Biagini {\it et al}~\cite{BHOZ}, Hu~\cite{Hu2},
Mishura~\cite{Mishura2}, Nualart~\cite{Nua4} and the references
therein. Throughout this paper we assume that $0<H<\frac12$ is
arbitrary but fixed and let $B^H=\{B_t^H, 0\leq t\leq T\}$ be a
one-dimensional fBm with Hurst index $H$ defined on $(\Omega,
\mathcal{F}, P)$. Let $({\mathcal S})^*$ be the Hida space of
stochastic distributions and let $\diamond$ denote the Wick product
on $({\mathcal S})^*$. Then $t\mapsto B_t^H$ is differentiable in
$({\mathcal S})^*$. Denote
$$
W^{(H)}_t=\frac{dB_t^H}{dt}\in ({\mathcal S})^*.
$$
We call $W^{(H)}$ the fractional white noise. For $u: {\mathbb
R}_+\to ({\mathcal S})^*$, in a white noise setting we define its
Wick-It\^o-Skorohod (WIS) stochastic integral with respect to $B^H$
by
\begin{equation}\label{eq2.100}
\int_0^tu_sdB^H_s:=\int_0^tu_s\diamond W^{(H)}_sds,
\end{equation}
whenever the last integral exists as an integral in $({\mathcal
S})^*$. We call these fractional It\^o integrals, because these
integrals share some properties of the classical It\^o integral. The
integral is closed in $L^2$, and moreover, for any $f\in
C^{2,1}({\mathbb R}\times [0,+\infty))$ the follwing It\^o type
formula holds:
\begin{align}\tag*{}
f(B_t^H,t)=f(0,0)+&\int_0^t\frac{\partial}{\partial x}
f(B_s^H,s)dB_s^H\\      \label{eq2.1}
&+\int_0^t\frac{\partial}{\partial
s}f(B_s^H,s)ds+H\int_0^t\frac{\partial^2}{\partial
x^2}f(B_s^H,s)s^{2H-1}ds.
\end{align}
The fBm $B^H$ has a local time ${\mathcal L}^{H}(x,t)$ continuous in
$(x,t)\in {\mathbb R}\times [0,\infty)$ which satisfies the
occupation formula (see Geman-Horowitz~\cite{Geman})
\begin{equation}\label{sec2-1-eq1}
\int_0^t\phi(B_s^{H},s)ds=\int_{\mathbb R}dx\int_0^t\phi(x,s)
{\mathcal L}^{H}(x,ds)
\end{equation}
for every continuous and bounded function $\phi(x,t):{\mathbb
R}\times {\mathbb R}_{+}\rightarrow {\mathbb R}$, and such that
$$
{\mathcal L}^{H}(x,t)=\int_0^t\delta(B_s^H-x)ds=\lim_{\epsilon
\downarrow
0}\frac{1}{2\epsilon}\lambda\big(s\in[0,t],|B_s^H-x|<\epsilon\big),
$$
where $\lambda$ denotes Lebesgue measure and $\delta(x)$ is the
Dirac delta function. Define the so-call weighted local time
${\mathscr L}^{H}(x,t)$ of $B^H$ at $x$ as follows
$$
{\mathscr L}^{H}(x,t)=2H\int_0^ts^{2H-1}{\mathcal L}^{H}(x,ds)\equiv
2H\int_0^t\delta(B_s^H-x)s^{2H-1}ds.
$$
Then the Tanaka formula
\begin{equation}\label{eq2.4}
|B_t^H-x|=|x|+\int_0^t{\rm sign}(B^H_s-x)dB^H_s+{\mathscr
L}^{H}(x,t)
\end{equation}
holds.

For $H\in (0,1)$ we define the operator ${M}$ on $L^2({\mathbb R})$
as follows (see Chapter 4 in Biagini {\it et al}~\cite{BHOZ} and
Elliott-Van der Hoek~\cite{Elli}):
$$
{M}f(x)=-\frac{\beta_H}{H-\frac12}\frac{d}{dx}\int_{\mathbb
R}\frac{(s-x)}{|s-x|^{\frac32-H}}f(s)ds,\qquad f\in L^2({\mathbb
R}),
$$
where $\beta_H$ is a normalizing constant. In particular, for
$H=\frac12$ we have ${M}f(x)=f(x)$, and for $0<H<\frac12$ we have
$$
{M}f(x)=\beta_H\int_{\mathbb
R}\frac{f(x-s)-f(x)}{|s|^{\frac32-H}}ds.
$$
As an example let us recall ${M}{1}_{[a,b]}(x)$, i.e., ${M}f$ when
$f$ is the indicator function of an interval $[a, b]$ with $a<b$. By
Elliott-Van der Hoek~\cite{Elli}, ${M}{1}_{[a,b]}(x)$ can be
calculated explicitly as
\begin{equation}\label{sec2-eq2.5-100}
{M}{1}_{[a,b]}(x)=\frac{\sqrt{\Gamma(2H+1)\sin(\pi
H)}}{2\Gamma(H+\frac12)\cos\left(\frac{\pi}2(H+\frac12)\right)}
\left(\frac{b-x}{|b-x|^{\frac32-H}}
-\frac{a-x}{|a-x|^{\frac32-H}}\right).
\end{equation}
By using the operator $M$ we can give the relation between
fractional and classical white noise (see Chapter 4 in Biagini {\it
et al}~\cite{BHOZ})
$$
W^{(H)}_t={M}W_t,
$$
which leads to
$$
\int_0^Tu_tdB^H_t=\int_{\mathbb R}M\left(u{1}_{[0,T]}\right)_t\delta
B_t,
$$
where $u$ is an adapted process and $\int_{\mathbb R}v_t\delta B_t$
denotes the Skorohod integral with respect to Brownian motion $B$
defined by
$$
\int_{\mathbb R}v_t\delta B_t:=\int_{\mathbb R}v_t\diamond W_tdt.
$$
Let $D^{(H)}_t$ denotes the Hida-Malliavin derivative with respect
to $B^H$. In the classical case $(H=1/2)$ we use the notation $D_t$
for the corresponding Hida-Malliavin derivative (for further
details, see Nualart~\cite{Nua4} and Biagini {\it et
al}~\cite{BHOZ}). We have
$$
D_tF={M}D^{(H)}_tF
$$
and
\begin{equation}\label{sec2-eq2.5-1000}
E\left[F\int_0^Tu_sdB^H_s\right]=E\left[\int_{\mathbb
R}({M}u{1}_{[0,T]})_s({M}D^{(H)}_sF)ds\right]
\end{equation}
for $F\in L^2(P)$.


\section{Some basic estimates}~\label{sec3}
In this section we will establish some technical estimates
associated with fractional Brownian motion with $0<H<\frac12$. For
simplicity throughout this paper we let $C$ stand for a positive
constant depending only on the subscripts and its value may be
different in different appearance, and this assumption is also
adaptable to $c$.

\begin{lemma}\label{lem3.1}
For all $t,s\in [0,T],\;t\geq s$ and $0<H<1$ we have
\begin{equation}\label{eq3.1}
\frac12(2-2^H)s^{2H}(t-s)^{2H}\leq t^{2H}s^{2H}-\mu^2\leq
2s^{2H}(t-s)^{2H},
\end{equation}
where $\mu=E(B^H_tB^H_s)$.
\end{lemma}

By the local nondeterminacy of fBm we can prove the lemma. Here, we
shall use an elementary method to prove it. We shall use the
following inequalities:
\begin{align}\label{eq3.3000-1}
(1+x)^\alpha&\leq 1+(2^\alpha-1)x^\alpha\\ \label{sec3-eq3.6=0}
(2-2^{\alpha})x^{\alpha}(1-x)^{\alpha}&\leq
(1-x)^{\alpha}-(1-x^{\alpha})\leq x^{\alpha}(1-x)^{\alpha}
\end{align}
with $0\leq x,\alpha\leq 1$. The inequality~\eqref{eq3.3000-1} is a
calculus exercise, and it is stronger than the well known
(Bernoulli) inequality
$$
(1+x)^\alpha\leq 1+\alpha x^\alpha\leq 1+x^\alpha,
$$
because $2^\alpha-1\leq \alpha$ for all $0\leq \alpha\leq 1$. The
inequalities~\eqref{sec3-eq3.6=0} are the improvement of the
classical inequality
$$
1-x^{\alpha}\leq (1-x)^{\alpha}.
$$
The right inequality in~\eqref{sec3-eq3.6=0} follows from the fact
$$
(1-x)^{\alpha}(1-x^{\alpha})\leq 1-x^{\alpha}.
$$
For the left inequality in~\eqref{sec3-eq3.6=0},
by~\eqref{eq3.3000-1} we have
\begin{align*}
1=(1-x+x)^{\alpha}&\leq (1-x)^{\alpha}\vee
x^{\alpha}+(2^{\alpha}-1)\left[(1-x)^{\alpha}\wedge
x^{\alpha}\right]
\end{align*}
for $0\leq x\leq 1$, where $x\vee y=\max\{x,y\}$ and $x\wedge
y=\min\{x,y\}$, which deduces
\begin{align*}
(1-x)^{\alpha}-(1-x^{\alpha})&\geq (2-2^{\alpha})(1-x)^{\alpha}
\wedge x^{\alpha}\\
&\geq (2-2^{\alpha})(1-x)^{\alpha}x^{\alpha}.
\end{align*}

\begin{proof}[Proof of~\eqref{eq3.1}]
Take $s=xt, 0\leq x\leq 1$. Then we can rewrite
$\rho_{r,s}:=t^{2H}s^{2H}-\mu^2$ as
\begin{align*}
\rho_{r,s}&=t^{4H}\left\{x^{2H}-\frac14\left[1+x^{2H}-
(1-x)^{2H}\right]^2\right\}\\
&\equiv t^{4H}G(x).
\end{align*}
In order to show the lemma we claim that
\begin{equation}\label{eq3.2}
\frac12(2-2^H)x^{2H}(1-x)^{2H}\leq G(x)\leq 2x^{2H}(1-x)^{2H}
\end{equation}
for all $x\in [0,1]$. We have
\begin{align*}
G(x)&=x^{2H}-\frac14\left[1+x^{2H}- (1-x)^{2H}\right]^2\\
&=\frac14\left\{2x^{H}-\left(1+x^{2H}- (1-x)^{2H}\right)\right\}
\left\{2x^{H}+\left(1+x^{2H}- (1-x)^{2H}\right)\right\}\\
&=\frac14\left\{(1-x)^{2H}-(1-x^{H})^2\right\}
\left\{2x^{H}+x^{2H}+1- (1-x)^{2H}\right\}\\
&=\frac14\left\{(1-x)^{H}-(1-x^{H})\right\}
\left\{(1-x)^{H}+1-x^{H}\right\}\left\{2x^{H}+x^{2H}+1-
(1-x)^{2H}\right\}.
\end{align*}
Thus,~\eqref{eq3.2} follows from~\eqref{sec3-eq3.6=0} and the facts
\begin{align*}
(1-x)^{H}\leq (1-x)^{H}&+(1-x^{H})\leq 2(1-x)^H,\\
2x^H\leq 2x^{H}+x^{2H}+1-& (1-x)^{2H}\leq 4x^{H}.
\end{align*}
This completes the proof.
\end{proof}

\begin{lemma}\label{lem3.2}
For all $t,s\in [0,T],\;t\geq s$ and $0<H<\frac12$ we have
\begin{equation}\label{eq3.3}
\frac12(t-s)^{2H}\leq t^{2H}-\mu\leq (t-s)^{2H},
\end{equation}
and
\begin{equation}\label{eq3.4}
\frac12(2-2^H)(\frac{s}{t})^{2H}(t-s)^{2H}\leq s^{2H}-\mu\leq
\frac12(\frac{s}{t})^{2H}(t-s)^{2H},
\end{equation}
where $\mu=E(B^H_tB^H_s)$.
\end{lemma}
\begin{proof}
The inequalities~\eqref{eq3.3} follow from
\begin{align*}
t^{2H}-\mu&=t^{2H}-\frac12\left(t^{2H}+s^{2H}-(t-s)^{2H}\right)\\
&=\frac12\left(t^{2H}-s^{2H}\right)+\frac12(t-s)^{2H}.
\end{align*}
In order to show that~\eqref{eq3.4}, we have
\begin{align*}
s^{2H}-\mu&=s^{2H}-\frac12\left(t^{2H}+s^{2H}-(t-s)^{2H}\right)\\
&=\frac12t^{2H}\left\{\left(1-\frac{s}{t}\right)^{2H}
-\left(1-(\frac{s}{t})^{2H}\right) \right\}.
\end{align*}
Thus, the inequalities~\eqref{eq3.4} follow
from~\eqref{sec3-eq3.6=0}. This completes the proof.
\end{proof}
\begin{lemma}\label{lem3.5}
For $0<H<\frac12$ we have
\begin{equation}\label{eq3.6}
\left|E\left[(B^H_t-B^H_s)(B^H_{t'}-B^H_{s'})\right]\right|\leq
C_H\frac{(t-s)^{2H}(t'-s')^{2H}}{(s-t')^{2H}}
\end{equation}
for all $0<s'<t'<s<t$.
\end{lemma}
Moreover, the estimate~\eqref{eq3.6} holds also for all
$0<s'<s<t'<t$. In fact we have
\begin{align*}
(t'-s)^{4H}&=(t'-s)^{2H}(t'-s)^{2H}\leq (t-s)^{2H}(t'-s')^{2H},\\
(t-t')^{2H}(t'-s)^{2H}&\leq (t-s)^{2H}(t'-s')^{2H},\\
(s-s')^{2H}(t'-s)^{2H}&\leq (t'-s')^{2H}(t-s)^{2H},\\
(t-s')^{2H}&=\left\{(t-s)+(s-s')\right\}^{2H}\leq
(t-s)^{2H}+(s-s')^{2H}\\
&=\frac{(t-s)^{2H}(t'-s)^{2H}+(s-s')^{2H}(t'-s)^{2H}}{(t'-s)^{2H}}\\
&\leq 2\frac{(t-s)^{2H}(t'-s')^{2H}}{(t'-s)^{2H}},
\end{align*}
which gives
\begin{align*}
|E\left[(B^H_t-B^H_s)(B^H_{t'}-B^H_{s'})\right]|&=\frac12\left\{
|t-s'|^{2H}+|s-t'|^{2H}-|t-t'|^{2H}-|s-s'|^{2H}\right\}\\
&\leq 3 \frac{(t-s)^{2H}(t'-s')^{2H}}{(t'-s)^{2H}}.
\end{align*}

\begin{proof}[Proof of~\eqref{eq3.6}]
For $0<s'<t'<s<t\leq T$ we define the function $x\mapsto G_{s,t}(x)$
on $[s',t']$ by
$$
G_{s,t}(x)=(s-x)^{2H}-(t-x)^{2H}.
$$
Thanks to mean value theorem, we see that there are $\xi\in (s',t')$
and $\eta\in (s,t)$ such that
\begin{align*}
2E\left[(B^H_t-B^H_s)(B^H_{t'}-B^H_{s'})\right]
&=G_{s,t}(t')-G_{s,t}(s')\\
&=2H(t'-s')\left[(t-\xi)^{2H-1}-(s-\xi)^{2H-1}\right]\\
&=2H(2H-1)(t'-s')(t-s)\left(\eta-\xi\right)^{2H-2}\leq 0,
\end{align*}
which gives
\begin{align}\label{eq3-00000}
|E\left[(B^H_t-B^H_s)(B^H_{t'}-B^H_{s'})\right]|\leq
\frac{(t'-s')(t-s)}{(s-t')^{2-2H}}.
\end{align}
On the other hand, noting that
$$
\frac{|E\left[(B^H_t-B^H_s)(B^H_{t'}-B^H_{s'})\right]|}{
(t-s)^H(t'-s')^H} \leq 1,
$$
we see that
\begin{align*}
&\frac{|E[(B^H_t-B^H_s)(B^H_{t'}-B^H_{s'})]|}{(t-s)^H(t'-s')^H}\leq
\left(\frac{|E\left[(B^H_t-B^H_s)(B^H_{t'}-B^H_{s'})\right]|}{
(t-s)^H(t'-s')^H}\right)^\alpha
\end{align*}
for all $\alpha\in [0,1]$. Combining this with~\eqref{eq3-00000}, we
get
\begin{align*}
|E[(B^H_t-B^H_s)&(B^H_{t'}-B^H_{s'})]|\leq
\frac{(t-s)^{(1-\alpha)H+\alpha}
(t'-s')^{(1-\alpha)H+\alpha}}{(s-t')^{\alpha(2-2H)}},
\end{align*}
and the lemma follows by taking $\alpha=H/(1-H)$.
\end{proof}
\begin{lemma}\label{lem3.6}
For $0<H<\frac12$ we have
\begin{align*}
&\left|E\left[B^H_t(B^H_{t}-B^H_{s})\right]\right|\leq (t-s)^{2H},\\
&\left|E\left[B^H_t(B^H_{s}-B^H_{r})\right]\right|\leq (s-r)^{2H},\\
&\left|E\left[B^H_r(B^H_{t}-B^H_{s})\right]\right|\leq (t-s)^{2H}
\end{align*}
for all $t>s>r>0$.
\end{lemma}

Let $\varphi(x,y)$ be the density function of $(B^H_s,B^H_r)$
($s>r>0$). That is
$$
\varphi(x,y)=\frac1{2\pi\rho}\exp\left\{-\frac{1}{2\rho^2}\left(
r^{2H}x^2-2\mu xy+s^{2H}y^2\right)\right\},
$$
where $\mu=E(B^H_sB^H_r)$ and $\rho^2=r^{2H}s^{2H}-\mu^2$.
\begin{lemma}\label{lem3.3}
Let $f\in C^1({\mathbb R})$ admit compact support. Then we have
\begin{align*}
|E\left[f'(B^H_{s})f'(B^H_{r})\right]|&\leq
\frac{C_Hs^H}{r^{H}(s-r)^{2H}}\left(E\left[|f(B^H_s)|^2 \right]
E\left[|f(B^H_r)|^2\right]\right)^{1/2}
\end{align*}
for all $s>r>0$ and $0<H<\frac12$.
\end{lemma}
\begin{proof}
Elementary calculation shows that
\begin{align*}
\int_{\mathbb{R}^2}f^2(y)&(x-\frac{\mu}{r^{2H}}
y)^2\varphi(x,y)dxdy\\
&=\frac{\rho^2}{r^{2H}}\int_{\mathbb{R}}f^2(y)
\frac{1}{\sqrt{2\pi}r^{H}}
e^{-\frac{y^2}{2r^{2H}}}dy=\frac{\rho^2}{r^{2H}}
E\left[|f(B^H_r)|^2\right],
\end{align*}
which implies that
\begin{align*}
\frac1{\rho^4}\int_{\mathbb{R}^2}|f(x)f(y)
(s^{2H}y-&\mu x)(r^{2H}x-\mu y)|\varphi(x,y)dxdy\\
&\leq \frac{r^{H}s^{H}}{\rho^2}\left(
E\left[|f(B^H_s)|^2\right]E\left[|f(B^H_r)|^2\right]\right)^{1/2}\\
&\leq \frac{C_Hs^H}{r^{H}(s-r)^{2H}}\left(
E\left[|f(B^H_s)|^2\right]E\left[|f(B^H_r)|^2\right]\right)^{1/2}
\end{align*}
by Lemma~\ref{lem3.1}. It follows that
\begin{align*}
|E[f'(B^H_{s})&f'(B^H_{r})]|=|\int_{\mathbb{R}^2}
f(x)f(y)\frac{\partial^{2}}{\partial x\partial
y}\varphi(x,y)dxdy|\\
&=|\int_{\mathbb{R}^2} f(x)f(y)\left\{\frac1{\rho^4}(s^{2H}y-\mu
x)(r^{2H}x-\mu y)+\frac{\mu}{\rho^2}\right\}\varphi(x,y)dxdy|\\
&\leq \frac{C_Hs^H}{r^{H}(s-r)^{2H}}\left(
E\left[|f(B^H_s)|^2\right]E\left[|f(B^H_r)|^2\right]\right)^{1/2}.
\end{align*}
This completes the proof.
\end{proof}

\begin{lemma}\label{lem3.4}
Let $f\in C^2({\mathbb R})$ admit compact support. Then we have
\begin{align*}
|E\left[f''(B^H_{s})f(B^H_{r})\right]|&\leq
\frac{C_H}{(s-r)^{2H}}\left( E\left[|f(B^H_s)|^2\right]E
\left[|f(B^H_r)|^2\right]\right)^{1/2}
\end{align*}
for all $s>r>0$ and $0<H<\frac12$.
\end{lemma}
\begin{proof}
A straightforward calculation shows that
\begin{align*}
\int_{\mathbb{R}^2}f^2(y)(x-\frac{\mu}{r^{2H}} y)^4\varphi(x,y)dxdy
&=\frac{3\rho^4}{r^{4H}}\int_{\mathbb{R}}f^2(y)
\frac{1}{\sqrt{2\pi}r^{H}} e^{-\frac{y^2}{2r^{2H}}}dy,
\end{align*}
which deduces
\begin{align*}
\frac1{\rho^4}\int_{\mathbb{R}^2}f(x)f(y)
(r^{2H}x-\mu y)^2&\varphi(x,y)dxdy\\
&\leq \frac{C_H}{(s-r)^{2H}}\sqrt{E\left[|f(B^H_s)|^2\right]
E\left[|f(B^H_r)|^2\right]}
\end{align*}
by Cauchy's inequality and Lemma~\ref{lem3.1}. It follows that
\begin{align*}
|E[f''(B^H_{s})&f(B^H_{r})]|=|\int_{\mathbb{R}^2}
f(x)f(y)\frac{\partial^{2}}{\partial x^2}
\varphi(x,y)dxdy|\\
&=|\int_{\mathbb{R}^2} f(x)f(y)\left\{\frac1{\rho^4}(r^{2H}x-\mu
y)^2
-\frac{r^{2H}}{\rho^2}\right\}\varphi(x,y)dxdy|\\
&\leq \frac{C_H}{(s-r)^{2H}}\left(
E\left[|f(B^H_s)|^2\right]E\left[|f(B^H_r)|^2\right]\right)^{1/2}.
\end{align*}
This completes the proof.
\end{proof}

\section{Existence of the generalized quadratic covariation}~\label{sec4}

In this section, for $0<H<\frac12$ we study the existence of the
{\it generalized quadratic covariation}. Denote
$$
J_\varepsilon(f,t):=\frac{1}{\varepsilon^{2H}}\int_0^t\left\{f(B^{H}_{
s+\varepsilon}) -f(B^{H}_s)\right\}(B^{H}_{s+\varepsilon}-B^{H}_s)ds^{2H}
$$
for $\varepsilon>0$ and $t\geq 0$. Recall that the {\it generalized
quadratic covariation} $[f(B^H),B^H]^{(W)}_t$ is defined as
\begin{equation}\label{sec4-eq1.1}
[f(B^H),B^H]^{(W)}_t:=\lim_{\varepsilon\downarrow
0}J_\varepsilon(f,t),
\end{equation}
provided the limit exists uniformly in probability. Clearly, we
have (see, for example, Klein and Gin\'e~\cite{Klein})
\begin{align}
[B^H,B^H]^{(W)}_t=t^{2H}
\end{align}
for all $t\geq 0$. In fact, one can easily prove that
\begin{align*}
E&\left|\frac{1}{\varepsilon^{2H}}\int_0^t
(B^H_{s+\varepsilon}-B^H_s)^2ds-t^{2H}\right|^2\\
&\qquad= \frac{1}{\varepsilon^{4H}}\int_0^t\int_0^t E\left[
(B^{H}_{r+\varepsilon}-B^{H}_r)^2(B^{H}_{s+\varepsilon}-B^{H}_s)^2
\right]ds^{2H}dr^{2H}-t^{4H}\\
&\qquad\longrightarrow 0
\end{align*}
for $t\geq 0$, as $\varepsilon\downarrow 0$.

Consider the decomposition
\begin{equation}\label{sec4-eq4.000000}
\begin{split}
&\frac{1}{\varepsilon^{2H}}\int_0^t\left\{f(B^{H}_{
s+\varepsilon})-f(B^{H}_s)\right\}(B^{H}_{s+\varepsilon}-B^{H}_s)ds^{2H}\\
&=\frac{1}{\varepsilon^{2H}}\int_0^tf(B^{H}_{
s+\varepsilon})(B^{H}_{s+\varepsilon}-B^{H}_s)ds^{2H}-\frac{1}{\varepsilon^{2H}}
\int_0^tf(B^{H}_s)(B^{H}_{s+\varepsilon}-B^{H}_s)ds^{2H}\\
&\equiv I_\varepsilon^{+}(f,t)-I_\varepsilon^{-}(f,t),
\end{split}
\end{equation}
and define the set ${\mathscr H}=\{f\,:\,{\text { measurable
functions on ${\mathbb R}$ such that $\|f\|_{\mathscr
H}<\infty$}}\}$, where
\begin{align*}
\|f\|_{\mathscr H}:=&\sqrt{\int_0^T\int_{\mathbb
R}|f(x)|^2e^{-\frac{x^2}{2s^{2H}}}\frac{dxds}{\sqrt{2\pi}s^{1-H}}}
+\sqrt{\int_0^T\int_{\mathbb R}|f(x)|^2
e^{-\frac{x^2}{2s^{2H}}}\frac{dxds}{\sqrt{2\pi}(T-s)^{1-H}}}.
\end{align*}
Then, ${\mathscr H}$ is a Banach space and the set ${\mathscr E}$ of
elementary functions of the form
$$
f_\triangle(x)=\sum_{i}f_{i}1_{(x_{i-1},x_{i}]}(x)
$$
is dense in ${\mathscr H}$, where $\{x_i,0\leq i\leq l\}$ is an
finite sequence of real numbers such that $x_i<x_{i+1}$. Moreover,
${{\mathscr H}}$ contains the sets ${\mathscr H}_\gamma$,
$\gamma>2$, of measurable functions $f$ such that
$$
\int_{0}^T\int_{\mathbb
R}|f(x)|^{\gamma}e^{-\frac{x^2}{2s^{2H}}}\frac{dxds}{\sqrt{2\pi}s^{1-H}}
<\infty.
$$

Our main object of this section is to explain and prove the following theorem.
\begin{theorem}\label{th4.1}
Let $0<H<\frac12$ and $f\in {\mathscr H}$. Then the generalized
quadratic covariation $[f(B^H),B^H]^{(W)}$ exists and
\begin{align}\label{th4.1-eq}
E\left|[f(B^H),B^H]^{(W)}_t\right|^2\leq C_H\|f\|_{\mathscr H}^2.
\end{align}
\end{theorem}

We split the proof into several lemmas, and for simplicity
throughout this paper we let $T=1$.

\begin{lemma}\label{lem4.1}
Let $0<H<\frac12$ and let $f$ be an infinitely differentiable function with compact support. We then have
\begin{align}
&E\left|I_\varepsilon^{-}(f,t)\right|^2\leq C_H\|f\|_{\mathscr H}^2,\\
&E\left|I_\varepsilon^{+}(f,t)\right|^2\leq C_H\|f\|_{\mathscr H}^2
\end{align}
for all $0<\varepsilon\leq 1$.
\end{lemma}
\begin{proof}
We need only to obtain the first estimate. It follows from~\eqref{sec2-eq2.5-1000}
that
\begin{align*}
E&\left[f(B^{H}_{
s})f(B^{H}_{r})(B^{H}_{s+\varepsilon}-B^{H}_s)(B^{H}_{r+\varepsilon}-B^{H}_r)\right]\\
&=E\left[f(B^{H}_{
s})f(B^{H}_{
r})(B^{H}_{s+\varepsilon}-B^{H}_s)\int_r^{r+\varepsilon}dB^{H}_l\right]\\
&=E\int_{\mathbb R}M1_{[r,r+\varepsilon]}(l)MD_l^{(H)}
f(B^{H}_{s})f(B^{H}_{r})(B^{H}_{s+\varepsilon}-B^{H}_s)dl\\
&=\int_{\mathbb R}M1_{[r,r+\varepsilon]}(l)M1_{[0,s]}(l)E\left[
f'(B^H_{s})f(B^H_{r})(B^{H}_{s+\varepsilon}-B^{H}_s) \right]dl\\
&\qquad+\int_{\mathbb R}M1_{[r,r+\varepsilon]}(l)M1_{[0,r]}(l)E\left[
f(B^H_{s})f'(B^H_{r})(B^{H}_{s+\varepsilon}-B^{H}_s) \right]dl\\
&\qquad\qquad+\int_{\mathbb
R}M1_{[r,r+\varepsilon]}(l)M1_{[s,s+\varepsilon]}(l)E\left[
f(B^H_{s})f(B^H_{r})\right]dl\\
&=E\left[B^H_{s}(B^{H}_{r+\varepsilon}-B^{H}_r)\right]E\left[
f'(B^H_{s})f(B^H_{r})(B^{H}_{s+\varepsilon}-B^{H}_s) \right]\\
&\qquad +E\left[B^H_{r}(B^{H}_{r+\varepsilon}-B^{H}_r)\right]E\left[
f(B^H_{s})f'(B^H_{r})(B^{H}_{s+\varepsilon}-B^{H}_s) \right]\\
&\qquad\qquad +E\left[(B^{H}_{r+\varepsilon}-B^{H}_r)
(B^{H}_{s+\varepsilon}-B^{H}_s)\right]E\left[
f(B^H_{s})f(B^H_{r})\right]\\
&\equiv \Psi_{\varepsilon}(s,r,1)+\Psi_{\varepsilon}(s,r,2)
+\Psi_{\varepsilon}(s,r,3).
\end{align*}
In order to end the proof we claim now that
\begin{equation}\label{eq4.4}
\frac{1}{\varepsilon^{4H}}\left|\int_0^t\int_0^t\Psi_{\varepsilon}
(s,r,k)ds^{2H}dr^{2H} \right|\leq C_H\|f\|^2_{\mathscr H},\qquad
k=1,2,3,
\end{equation}
for all $\varepsilon>0$ small enough. Some elementary calculus can
show that, for all $0<\varepsilon\leq 1$
\begin{align*}
\int_{\varepsilon}^1E\left[|f(B^H_{s})|^2\right]&s^{2H-1}ds
\int_0^{s-\varepsilon} \frac{dr}{r^{1-2H}(s-\varepsilon-r)^{2H}}\\
&=\int_{\varepsilon}^1E\left[|f(B^H_{s})|^2\right] s^{2H-1}ds
\int_0^{s-\varepsilon} \frac{dr}{r^{1-2H}(s-\varepsilon-r)^{2H}}\\
&=\int_{\varepsilon}^1s^{2H-1} E\left[|f(B^H_{s})|^2
\right]ds\left(\int_0^1 \frac{dr}{x^{1-2H}(1-x)^{2H}}dx\right),\\
\int_{\varepsilon}^1E\left[|f(B^H_{s})|^2 \right]&s^{2H-1}ds
\int_{s-\varepsilon}^s\frac{dr}{r^{1-2H}(r+\varepsilon-s)^{2H}}\\
&\leq \int_{\varepsilon}^1E\left[|f(B^H_{s})|^2 \right]ds
\int_{s-\varepsilon}^s\frac{dr}{r^{2-4H}(r+\varepsilon-s)^{2H}}\\
&=\int_{\varepsilon}^1E\left[|f(B^H_{s})|^2 \right]ds
\int_1^{\frac{s}{s-\varepsilon}}
\frac{dx}{x^{2-4H}(x-1)^{2H}}\\
&\leq \int_0^1E\left[|f(B^H_{s})|^2 \right]ds\left(\int_1^{+\infty}
\frac{dx}{x^{2-4H}(x-1)^{2H}}\right),
\end{align*}
and
\begin{align}\notag
\int_0^{\varepsilon}E\left[|f(B^H_{s})|^2 \right]&s^{2H-1}ds\int_0^s
\frac{r^{2H-1}dr}{(r+\varepsilon-s)^{2H}}\\  \notag
&=\int_0^{\varepsilon}E\left[|f(B^H_{s})|^2
\right]s^{2H-1}ds\int_0^{\frac{s}{\varepsilon-s}}
\frac{x^{2H-1}dx}{(1+x)^{2H}}\\ \notag &\leq
\int_0^{\varepsilon}E\left(|f(B^H_{s})|^2
\right)ds\frac{s^{3H-1}}{(\varepsilon-s)^H}\left( \int_0^{\infty}
\frac{x^{H-1}dx}{(1+x)^{2H}}\right)\\   \label{eq4.1000}&\leq {C_H}
\int_0^1\int_{\mathbb
R}|f(x)|^2e^{-\frac{x^2}{2s^{2H}}}\frac{s^{2H-1}dxds}{
{\sqrt{2\pi}}(1-s)^H},
\end{align}
where the estimate~\eqref{eq4.1000} follows from the monotonicity of
the function
$$
\varepsilon\mapsto
\int_0^{\varepsilon}\frac{s^{2H-1}}{(\varepsilon-s)^H}
e^{-\frac{x^2}{2s^{2H}}}ds
$$
with $\varepsilon\in [0,1]$. It follows that
\begin{align*}
\frac{1}{\varepsilon^{4H}}&\left|\int_0^1\int_0^1
\Psi_{\varepsilon}(s,r,3) ds^{2H}dr^{2H}\right|\\
&\leq \frac{H}{\varepsilon^{4H}}\int_0^1\int_0^1
|E\left[(B^{H}_{r+\varepsilon}-B^{H}_r)
(B^{H}_{s+\varepsilon}-B^{H}_s)\right]|\\
&\qquad\qquad\qquad\cdot\left\{E\left[
f^2(B^H_{s})\right]+E\left[f^2(B^H_{r})\right]\right\}(sr)^{2H-1}dsdr\\
&= \frac{H}{\varepsilon^{4H}}\int_0^1\int_0^1
|E\left[(B^{H}_{r+\varepsilon}-B^{H}_r)
(B^{H}_{s+\varepsilon}-B^{H}_s)\right]|E\left[
f^2(B^H_{s})\right](sr)^{2H-1}dsdr\\
&\leq H\int_{\varepsilon}^1E\left[|f(B^H_{s})|^2\right]s^{2H-1}ds
\int_0^{s-\varepsilon} \frac{dr}{r^{1-2H}(s-\varepsilon-r)^{2H}}\\
&\qquad\qquad+H\int_{\varepsilon}^1E\left[|f(B^H_{s})|^2
\right]s^{2H-1}ds
\int_{s-\varepsilon}^s\frac{dr}{r^{1-2H}(r+\varepsilon-s)^{2H}}\\
&\qquad\qquad+H\int_0^{\varepsilon}E\left[|f(B^H_{s})|^2
\right]s^{2H-1}ds\int_0^s
\frac{r^{2H-1}dr}{(r+\varepsilon-s)^{2H}}\\
&\leq C_H\|f\|_{\mathscr H}^2
\end{align*}
for all $0<\varepsilon\leq 1$.

Now, let us obtain the estimate~\eqref{eq4.4} for $k=1$.
By~\eqref{sec2-eq2.5-1000} we see that
\begin{align*}
\Psi_{\varepsilon}(s,r,1)&=E\left[B^H_{s}(B^{H}_{r+\varepsilon}-B^{H}_r)\right]E\left[
f'(B^H_{s})f(B^H_{r})(B^{H}_{s+\varepsilon}-B^{H}_s) \right]\\
&=E\left[B^H_{s}(B^{H}_{r+\varepsilon}-B^{H}_r)\right]
E\left[B^H_{s}(B^{H}_{s+\varepsilon}-B^{H}_s)\right]
E\left[f''(B^H_{s})f(B^H_{r})\right]\\
&\quad+E\left[B^H_{s}(B^{H}_{r+\varepsilon}-B^{H}_r)\right]
E\left[B^H_{r}(B^{H}_{s+\varepsilon}-B^{H}_s)\right]
E\left[f'(B^H_{s})f'(B^H_{r})\right]\\
&\equiv\Psi_{\varepsilon}(s,r,1,1)+\Psi_{\varepsilon}(s,r,1,2).
\end{align*}
Together Lemma~\ref{lem3.3}, Lemma~\ref{lem3.4}, Lemma~\ref{lem3.6}
and the fact
\begin{align}\label{eq4.5}
E\left[f^2(B^H_{r})\right]&=\int_{\mathbb R}f^2(x)\frac1{\sqrt{2\pi}
r^H}e^{-\frac{x^2}{2r^{2H}}}dx\\   \tag*{} &\leq
\frac{s^H}{r^H}\int_{\mathbb R}f^2(x)\frac1{\sqrt{2\pi}
s^H}e^{-\frac{x^2}{2s^{2H}}}dx=\frac{s^H}{r^H}E\left[f^2(B^H_{s})
\right]
\end{align}
with $s\geq r>0$ lead to
\begin{align*}
\frac{1}{\varepsilon^{4H}}&\left|\int_0^t\int_0^t\Psi_{\varepsilon}(s,r,1,1) ds^{2H}dr^{2H}\right|\leq \int_0^t\int_0^t
\left|E\left[f''(B^H_{s})f(B^H_{r})\right]\right|ds^{2H}dr^{2H}\\
&\leq C_H\int_0^t\int_0^s
\frac{1}{(s-r)^{2H}}E|f(B^H_{s})f(B^H_{r})|ds^{2H}dr^{2H}\\
&\leq C_H\int_0^tE[f^2(B^H_{s})]ds^{2H}\int_0^s
\frac{s^{H/2}}{(s-r)^{2H}r^{H/2}}dr^{2H}\\
&\leq C_H\|f\|_{\mathscr H}^2,
\end{align*}
and
\begin{align*}
\frac{1}{\varepsilon^{4H}}&\left|\int_0^t\int_0^t\Psi_{\varepsilon}(s,r,1,2) ds^{2H}dr^{2H}\right|\leq \int_0^t\int_0^t
\left|E\left[f'(B^H_{s})f'(B^H_{r})\right]\right|ds^{2H}dr^{2H}\\
&\leq C_H\int_0^t\int_0^s\frac{s^{H}}{r^H(s-r)^{2H}}|E\left[f(B^H_{s})
f(B^H_{r})|\right]ds^{2H}dr^{2H}\\
&\leq C_H\|f\|_{\mathscr H}^2
\end{align*}
for all $\varepsilon>0$ and $t\geq 0$. Thus, we get
\begin{align*}
\frac{1}{\varepsilon^{4H}}&\left|\int_0^t\int_0^t\Psi_{\varepsilon}(s,r,1) ds^{2H}dr^{2H}\right|\leq
C_H\|f\|_{\mathscr H}^2.
\end{align*}
Similarly, we can also obtain the estimate~\eqref{eq4.4} for $k=2$,
and the lemma follows.
\end{proof}

Recently, Gradinaru-Nourdin~\cite{Grad3} introduced the following
perfect result:
\begin{exampless}[Theorem 2.1 in
Gradinaru--Nourdin~\cite{Grad3}] Assume that $H\in(0,1)$. Let
$f:{\mathbb R}\to {\mathbb R}$ be a function satisfying
\begin{equation}\label{eq4.2-Gradinaru--Nourdin}
|f(x)-f(y)|\leq C|x-y|^a(1+x^2+y^2)^b,\quad (C>0,0<a\leq 1,b>0),
\end{equation}
for all $x,y\in {\mathbb R}$, and let $\{Y_t:\;t\geq 0\}$ be a
continuous stochastic process. Then, as $\varepsilon\to 0$,
\begin{equation}
\int_0^tY_sf(\frac{B^H_{s+\varepsilon}-B^H_s}{\varepsilon^H})ds
\longrightarrow E[f(N)]\int_0^tY_sds,
\end{equation}
almost surely, uniformly in $t$ on each compact interval, where $N$
is a standard Gaussian random variable.
\end{exampless}
According to the theorem above we get the next lemma.
\begin{lemma}\label{lem4.4} Let $0<H<1$ and $f\in
C({\mathbb R})$. We then have
\begin{align}\label{sec4-eq4.11-0}
\lim_{\varepsilon\downarrow 0}\frac{1}{\varepsilon^{2H}}
\int_0^tf(B^{H}_s)(B^{H}_{s+\varepsilon}-B^{H}_s)^2ds^{2H}
=\int_0^tf(B^H_s)ds^{2H}
\end{align}
almost surely, for all $t\geq 0$.
\end{lemma}
As a direct consequence of Lemma~\ref{lem4.4}, for $f\in
C^1({\mathbb R})$ we have
\begin{equation}\label{sec4-eq4.15}
\left[f(B^H),B^H\right]^{(W)}_t=2H\int_0^tf'(B^H_s)s^{2H-1}ds
\end{equation}
for all $0<H<1$. In fact, the H\"older continuity of fBm $B^H$
yields
\begin{align*}
\lim_{\varepsilon\downarrow 0}\frac{1}{\varepsilon^{2H}}
\int_0^to(B^{H}_{s+\varepsilon}-B^{H}_s)(B^{H}_{s+\varepsilon}
-B^{H}_s)^2ds^{2H}=0
\end{align*}
almost surely. It follows that
\begin{align*}
\lim_{\varepsilon\downarrow 0}&\frac{1}{\varepsilon^{2H}}\int_0^t
\left\{f(B^{H}_{
s+\varepsilon})-f(B^{H}_s)\right\}(B^{H}_{s+\varepsilon}-B^{H}_s)
ds^{2H}\\
&=\lim_{\varepsilon\downarrow 0}\frac{1}{\varepsilon^{2H}}
\int_0^tf'(B^{H}_s)(B^{H}_{s+\varepsilon}-B^{H}_s)^2ds^{2H} =\int_0^tf'(B^H_s)ds^{2H}
\end{align*}
almost surely.

Now we can show our main result.
\begin{proof}[Proof of Theorem~\ref{th4.1}]
Given $f\in {\mathscr H}$. If $f\in C^1({\mathbb R})$, then the
theorem follows from the identity~\eqref{sec4-eq4.15} and the
follows estimate:
\begin{align*}
E\left(\int_0^tf'(B^H_s)s^{2H-1}ds\right)^2
&=\int_0^t\int_0^tE\left[f'(B^H_s) f'(B^H_r)\right](sr)^{2H-1}dsdr\\
&\leq C_H\int_0^t\int_0^s\frac{s^{\frac{7H}2-1}}{r^{1-\frac{H}2}
(s-r)^{2H}}E\left[f^2(B^H_s)\right]dsdr\\
&\leq C_H\int_0^t s^{2H-1} E\left[f^2(B^H_s)\right]ds\leq
C_H\|f\|^2_{\mathscr H}
\end{align*}
by Lemma~\ref{lem3.3} and~\eqref{eq4.5}. Let now $f\not\in
C^{\infty}_0({\mathbb R})$.

Consider the function $\zeta$ on ${\mathbb R}$ by
\begin{equation}
\zeta(x):=
\begin{cases}
ce^{\frac1{(x-1)^2-1}}, &{\text { $x\in (0,2)$}},\\
0, &{\text { otherwise}},
\end{cases}
\end{equation}
where $c$ is a normalizing constant such that $\int_{\mathbb
R}\zeta(x)dx=1$. Define the so-called mollifiers
\begin{equation}\label{sec4-eq00-4}
\zeta_n(x):=n\zeta(nx),\qquad n=1,2,\ldots
\end{equation}
and the sequence of smooth functions
\begin{align}\label{sec4-eq00-1}
f_n(x)&=\int_{\mathbb R}f(x-y)\zeta_n(y)dy=
\int_0^2f(x-\frac{y}n)\zeta(y)dy,\qquad n=1,2,\ldots
\end{align}
for all $x\in \mathbb R$. Then $\{f_n\}\subset C^{\infty}({\mathbb
R})\cap {\mathscr H}$ and $f_n$ converges to $f$ in ${\mathscr H}$,
as $n$ tends to infinity.

On the other hand, by Lemma~\ref{lem4.1} we have
\begin{align*}
P(|J_{\varepsilon_1}(f,t)-J_{\varepsilon_2}(f,t)|\geq \delta)&\leq
P\left(|J_{\varepsilon_1}(f-f_n,t)|\geq
\frac{\delta}3\right)+P\left(|J_{\varepsilon_2}(f-f_n,t)|\geq
\frac{\delta}3\right)\\
&\qquad+P\left(|J_{\varepsilon_1}(f_n,t)-J_{\varepsilon_2}(f_n,t)|\geq
\frac{\delta}3\right)\\
&\leq \frac{C_{H}}{\delta^2}\|f-f_n\|_{\mathscr
H}^2+P\left(|J_{\varepsilon_1}(f_n,t)-J_{\varepsilon_2}(f_n,t)|\geq
\frac{\delta}3\right)
\end{align*}
for all $n$ and $\delta,\varepsilon_1,\varepsilon_2>0$. Combining
this with
$$
\lim_{\varepsilon\downarrow
0}J_{\varepsilon}(f_n,t)=[f_n(B^H),B^H]^{(W)}_t=
2H\int_0^tf'_n(B^H_s)s^{2H-1}ds,\qquad n\geq 1
$$
in probability, we show that the generalized quadratic covariation
$[f(B^H),B^H]^{(W)}$ exists for $f\in {\mathscr H}$. Thus, the
estimate~\eqref{th4.1-eq} follows from Lemma~\ref{lem4.1}. This
completes the proof.
\end{proof}
\begin{corollary}\label{lem5.2}
Let $f,f_1,f_2,\ldots\in {\mathscr H}$. If $f_n\to f$ in ${\mathscr
H}$, as $n$ tends to infinity, then we have
$$
[f_n(B^H),B^H]^{(W)}_t\longrightarrow [f(B^H),B^H]^{(W)}_t
$$
in $L^2$ as $n\to \infty$.
\end{corollary}
\begin{proof}
The convergence follows from
$$
E\left|[f_n(B^H),B^H]^{(W)}_t-[f(B^H),B^H]^{(W)}_t\right|^2\leq
C_H\|f_n-f\|_{\mathscr H}^2\to 0,
$$
as $n$ tends to infinity.
\end{proof}
By using the above result, we immediately get an extension of
It\^{o} formula stated as follows.
\begin{theorem}\label{th4.2}
Let $0<H<\frac12$ and let $f\in {\mathscr H}$ be left continuous. If
$F$ is an absolutely continuous function with the derivative $F'=f$,
then the following It\^o type formula holds:
\begin{equation}\label{sec5-eq5.3}
F(B^H)=F(0)+\int_0^tf(B^H_s)dB^H_s+\frac12\left[f(B^H),B^H
\right]^{(W)}_t.
\end{equation}
\end{theorem}

Clearly, this is an analogue of F\"ollmer-Protter-Shiryayev's
formula (see Eisenbaum~\cite{Eisen1}, F\"ollmer {\it et
al}~\cite{Follmer}, Moret--Nualart~\cite{Moret},
Russo--Vallois~\cite{Russo2}, and the references therein). It is an
improvement in terms of the hypothesis on $f$ and it is also quite
interesting itself.
\begin{proof}[Proof of Theorem~\ref{th4.2}]
If $F\in C^2({\mathbb R})$, then this is It\^o's formula since
$$
\left[f(B^H),B^H\right]^{(W)}_t=2H\int_0^tf'(B^H_s)s^{2H-1}ds.
$$
For $F\not\in C^2({\mathbb R})$, by a localization argument we may
assume that the function $f$ is uniformly bounded. In fact, for any
$k\geq 0$ we may consider the set
$$
\Omega_k=\left\{\sup_{0\leq t\leq T}|B^H_t|<k\right\}
$$
and let $f^{[k]}$ be a measurable function such that $f^{[k]}=f$ on
$[-k,k]$ and such that $f^{[k]}$ vanishes outside. Then $f^{[k]}$ is
uniformly bounded and $f^{[k]}\in {\mathscr H}$ for every $k\geq 0$.
Set $\frac{d}{dx}F^{[k]}=f^{[k]}$ and $F^{[k]}=F$ on $[-k,k]$. If
the theorem is true for all uniformly bounded functions on
${\mathscr H}$, then we get the desired formula
$$
F^{[k]}(B^H_t)=F^{[k]}(0)+\int_0^t
f^{[k]}(B^H_s)dB^H_s+\frac12\left[f^{[k]}(B^H),B^H\right]^{(W)}_t
$$
on the set $\Omega_k$. Letting $k$ tend to infinity we deduce the
It\^o formula~\eqref{sec5-eq5.3} for all $f\in {\mathscr H}$ being
left continuous and locally bounded.

Let now $F'=f\in {\mathscr H}$ be uniformly bounded and left
continuous. For any positive integer $n$ we define
$$
F_n(x):=\int_{\mathbb R}F(x-{y})\zeta_n(y)dy,\quad x\in {\mathbb R},
$$
where $\zeta_n$, $n\geq 1$ are the mollifiers defined
by~\eqref{sec4-eq00-4}. Then $F_n\in C^\infty({\mathbb R})$ for all
$n\geq 1$ and the It\^{o} formula
\begin{equation}\label{sec3-eq3-Ito-1}
F_n(B_t^{H})=F_n(0)+\int_0^tf_n(B_s^{H})dB_s^{H}+
H\int_0^tf_n'(B_s^{H})s^{2H-1}ds
\end{equation}
holds for all $n\geq 1$, where $f_n=F_n'$. Moreover using Lebesgue's
dominated convergence theorem, one can prove that as $n\to \infty$,
for each $x$,
$$
F_n(x)\longrightarrow F(x),\quad f_n(x)\longrightarrow f(x),
$$
and $\{f_n\}\subset {\mathscr H}$, $f_n\to f$ in ${\mathscr H}$, as
$n$ tends to infinity. It follows that
\begin{align*}
2H\int_0^tf_n'(B_s^{H})s^{2H-1}ds=[f_n(B^H),B^H]^{(W)}_t
\longrightarrow \left[f(B^H),B^H\right]^{(W)}_t
\end{align*}
in $L^2$ by Corollary~\ref{lem5.2}, as $n$ tends to infinity. It
follows that
\begin{align*}
\int_0^tf_n(B_s^{H})dB_s^{H}&=F_n(B_t^{H})-F_n(0)-
\frac12[f_n(B^H),B^H]^{(W)}_t\\
&\longrightarrow F(B_t^{H})-F(0)-\frac12[f(B^H),B^H]^{(W)}_t
\end{align*}
in $L^2$, as $n$ tends to infinity. This completes the proof since
the integral is closed in $L^2$.
\end{proof}


\section{Integration with respect to the local time}\label{sec5}
In this section we assume that $0<H<\frac12$ and study the integral
\begin{equation}\label{sec6-eq6.1}
\int_{\mathbb R}f(x){\mathscr L}^{H}(dx,t),
\end{equation}
where $f$ is a deterministic function and
$$
{\mathscr L}^{H}(x,t)=2H\int_0^t\delta(B^H_s-x)s^{2H-1}ds
$$
is the weighted local time of fBm $B^{H}$. Recall that the quadratic
covariation $[f(B),B]$ of Brownian motion $B$ can be characterized
as
$$
[f(B),B]_t=-\int_{\mathbb R}f(x){\mathscr L}^{B}(dx,t),
$$
where $f$ is locally square integrable and ${\mathscr L}^{B}(x,t)$
is the local time of Brownian motion. This is called the Bouleau-Yor
identity. More works for this can be found in
Bouleau-Yor~\cite{Bouleau}, Eisenbaum~\cite{Eisen1}, F\"ollmer {\it
et al}~\cite{Follmer}, Feng--Zhao~\cite{Feng},
Peskir~\cite{Peskir1}, Rogers--Walsh~\cite{Rogers2},
Yang--Yan~\cite{Yan2}, and the references therein. However, this is
not true for fractional Brownian motion. For $\frac12<H<1$, Yan {\it
et al}~\cite{Yan3,Yan1} obtained the following Bouleau-Yor identity:
$$
[f(B^H),B^H]_t^{(W)}=-\int_{\mathbb R}f(x){\mathscr L}^{H}(dx,t).
$$
In this section we show that the identity above also holds for
$0<H<\frac12$.

Take $F(x)=(x-a)^{+}-(x-b)^{+}$. Then $F$ is absolutely continuous
with the derivative $F'=1_{(a,b]}\in {\mathscr H}$ being left
continuous and bounded, and the It\^o formula~\eqref{sec5-eq5.3}
yields
\begin{align*}
\left[1_{(a,b]}(B^H),B^H\right]^{(W)}_t&=2F(B^H_t)-2F(0)-
2\int_0^t1_{(a,b]}(B^H_s)dB^H_s\\
&={\mathscr L}^{H}(a,t)-{\mathscr L}^{H}(b,t)
\end{align*}
for all $t\in [0,1]$. Thus, the linearity property of generalized
quadratic covariation deduces the following result.
\begin{lemma}\label{lem6.1}
For any $f_\triangle(x)=\sum_jf_j1_{(a_{j-1},a_j]}(x)\in {\mathscr
E}$, the integral
$$
\int_{\mathbb R}f_\triangle(x){\mathscr
L}^{H}(dx,t):=\sum_jf_j\left[{\mathscr L}^{H}(a_j,t)-{\mathscr
L}^{H}(a_{j-1},t)\right]
$$
exists and
\begin{equation}\label{sec6-eq6.3}
\int_{\mathbb R}f_{\Delta}(x)\mathscr{L}^{H}(dx,t)=
-\left[f_\triangle(B^H),B^H\right]^{(W)}_t
\end{equation}
for all $t\in [0,1]$.
\end{lemma}
Thanks to the density of ${\mathscr E}$ in ${{\mathscr H}}$, we can
then extend the definition of integration with respect to $x\mapsto
{\mathscr L}^H(x,t)$ to the elements of ${\mathscr H}$ in the
following manner:
\begin{equation*}
\int_{\mathbb R}f(x){\mathscr L}^{H}(dx,t):=\lim_{n\to
\infty}\int_{\mathbb R}f_{\triangle,n}(x){\mathscr L}^{H}(dx,t)
\end{equation*}
in $L^2$ for $f\in {{\mathscr H}}$ provided $f_{\triangle,n}\to f$
in ${{\mathscr H}}$, as $n$ tends to infinity, where
$\{f_{\triangle,n}\}\subset {\mathscr E}$. The limit obtained does
not depend on the choice of the sequence $\{f_{\triangle,n}\}$ and
represents the integral of $f$  with respect to ${\mathscr L}^{H}$.
Together this and Corollary~\ref{lem5.2} lead to the Bouleau-Yor
identity
\begin{equation}\label{sec6-eq6.4}
\left[f(B^H),B^H\right]^{(W)}_t=-\int_{\mathbb R}f(x){\mathscr
L}^{H}(dx,t)
\end{equation}
for all $t\in [0,1]$.
\begin{corollary}
Let $0<H<\frac12$ and let $f,f_1,f_2,\ldots\in {\mathscr H}$. If
$f_n\to f$ in ${\mathscr H}$, as $n$ tends to infinity, we then have
\begin{align*}
\int_{\mathbb R}f_n(x){\mathscr L}^{H}(dx,t)\longrightarrow
\int_{\mathbb R}f(x){\mathscr L}^{H}(dx,t)
\end{align*}
in $L^2$, as $n$ tends to infinity.
\end{corollary}

According to Theorem~\ref{th4.2}, we get an analogue of
Bouleau-Yor's formula.
\begin{corollary}\label{cor6.1}
Let $0<H<\frac12$ and let $f\in {\mathscr H}$ be left continuous. If
$F$ is an absolutely continuous function with the derivative $F'=f$,
then the following It\^o type formula holds:
\begin{equation}\label{sec6-eq6.5}
F(B_t^{H})=F(0)+\int_0^tf(B_s^{H})dB_s^{H}-\frac12\int_{\mathbb
R}f(x){\mathscr L}^{H}(dx,t).
\end{equation}
\end{corollary}
Recall that if $F$ is the difference of two convex functions, then
$F$ is an absolutely continuous function with derivative of bounded
variation. Thus, the It\^o-Tanaka formula
\begin{align*}
F(B_t^H)&=F(0)+\int_0^tF^{'}(B_s^H)dB_s^H+\frac12\int_{\mathbb
R}{\mathscr L}^{H}(x,t)F''(dx)\\
&\equiv F(0)+\int_0^tF^{'}(B_s^H)dB_s^H-\frac12\int_{\mathbb
R}F'(x){\mathscr L}^{H}(dx,t)
\end{align*}
holds. This is given by Coutin {\it et al}~\cite{Cout} (see also Hu
{\it et al}~\cite{Hu1}).

\begin{remark}
{\rm

By the proof similar to Lemma 3.1 in
Gradinaru--Nourdin~\cite{Grad3}, one can obtain the following
convergence (see also Gradinaru--Nourdin~\cite{Grad}):
\begin{equation}\label{sec5-eq5.11-0}
\lim_{n\to \infty}\sum_{j=1}^{n}\left(\Lambda_j\right)^{2H-1}g(B^H_{t_j})
(B^H_{t_{j}}-B^H_{t_{j-1}})^{2}=\int_0^tg(B^H_s)s^{2H-1}ds
\end{equation}
almost surely, where $\pi_n=\{0=t_0<t_1 <\cdots< t_n=t\}$ denotes an
arbitrary partition of the interval $[0, t]$ with $\|\pi_n\|=
\sup_{j}(t_{j}-t_{j-1})\to 0$, $\Lambda_{j}
=\frac{t_j}{t_j-t_{j-1}}$ and $g\in C({\mathbb R})$. Thus, similar
to proof of Theorem~\ref{th4.1} we can show that the convergence
$$
2H\lim_{n\to
\infty}\sum_{j=1}^{n}\left(\Lambda_j\right)^{2H-1}\{f(B^H_{t_j})
-f(B^H_{t_{j-1}})\} (B^H_{t_j}-B^H_{t_{j-1}})=-\int_{\mathbb
R}f(x){\mathscr L}^H(dx,t)
$$
holds, which deduces
$$
2H\lim_{n\to
\infty}\sum_{j=1}^{n}\left(\Lambda_j\right)^{2H-1}\{f(B^H_{t_j})
-f(B^H_{t_{j-1}})\} (B^H_{t_j}-B^H_{t_{j-1}})=[f(B^H),B^H]^{(W)}_t,
$$
where $f\in {\mathscr H}$ and the limits are uniform in probability.

}
\end{remark}


\section{The time-dependent case}~\label{sec6}
In this section we consider the time-dependent case. For a
measurable function $f$ on ${\mathbb R}\times {\mathbb R}_{+}$ we
define the generalized quadratic covariation
$[f(B^H,\cdot),B^H]^{(W)}$ of $f(B^H,\cdot)$ and $B^H$ as follows
\begin{equation}\label{sec6-eq6.1-0}
[f(B^H,\cdot),B^H]^{(W)}_t:=\lim_{\varepsilon\downarrow
0}\frac{1}{\varepsilon^{2H}}\int_0^t\left\{f(B^{H}_{
s+\varepsilon},s+\varepsilon)
-f(B^{H}_s,s)\right\}(B^{H}_{s+\varepsilon}-B^{H}_s)ds^{2H}
\end{equation}
for $t\in [0,T]$, provided the limit exists uniformly in
probability. We prove the existence of the quadratic covariation.

Consider the set ${\mathscr H}_{\ast}$ of measurable functions $f$
on ${\mathbb R}\times {\mathbb R}_{+}$ such that the function
$t\mapsto f(\cdot,t)$ is continuous and $\|f\|_{{\mathscr
H}_{\ast}}<+ \infty$, where
\begin{align*}
\|f\|_{{\mathscr H}_{\ast}}=\sqrt{\int_0^T\int_{\mathbb
R}|f(x,s)|^2e^{-\frac{x^2}{2s^{2H}}}\frac{dxds}{\sqrt{2\pi}s^{1-H}}}
+\sqrt{\int_0^T\int_{\mathbb R}|f(x,s)|^2
e^{-\frac{x^2}{2s^{2H}}}\frac{dxds}{\sqrt{2\pi}(T-s)^{1-H}}}
\end{align*}
with $\varphi_s(x)=\frac1{\sqrt{2\pi}s^H}e^{-\frac{x^2}{2s^{2H}}}$.
Then ${\mathscr H}_{\ast}$ is a Banach space and the set ${\mathscr
E}_{\ast}$ of elementary functions of the form
\begin{equation}\label{sec6-eq6.4-0}
f_\triangle(x,t)=\sum_{i,j}f_{ij}1_{(x_{i-1},x_{i}]}(x)
1_{(s_{j-1},s_{j}]}(t)
\end{equation}
is dense in ${\mathscr H}_{\ast}$, where $\{x_i,0\leq i\leq n\}$ is
an finite sequence of real numbers such that $x_i<x_{i+1}$,
$\{s_j,0\leq j\leq m\}$ is a subdivision of $[0,T]$ and $(f_{ij})$
is a matrix of order $n\times m$. Moreover, ${\mathscr H}_{\ast}$
contains the set ${\mathscr H}_{\ast,\gamma}$ with $\gamma>2$ of
measurable functions $f$ on ${\mathbb R}$ such that
$$
\int_{0}^T\int_{\mathbb R}|f(x,s)|^{\gamma}e^{-\frac{x^2}{2s^{2H}}}
\frac{dxds}{\sqrt{2\pi}s^{1-H}} <\infty.
$$
As a corollary of Theorem A, we have
\begin{equation}\label{sec6-eq6.2-0}
\lim_{\varepsilon\downarrow
0}\frac{1}{\varepsilon^{2H}}\int_0^ts^{2H-1}g(B^{H}_{s},s)
(B^{H}_{s+\varepsilon}-B^{H}_s)^2ds=\int_0^tg(B^{H}_s,s)s^{2H-1}ds
\end{equation}
almost surely, for all $t\geq 0$ if $g$ is continuous. This proves
the following identity:
\begin{equation}\label{sec6-eq6.3-0}
\left[f(B^{H},\cdot),B^{H}\right]^{(W)}_t=2H\int_0^t\frac{\partial
f}{\partial x}(B^{H}_s,s)s^{2H-1}ds
\end{equation}
for all $t\geq 0$, provided $f\in C^{1,1}({\mathbb R}\times{\mathbb
R}_+)$. Thus, similar to proof of Theorem~\ref{th4.1}, one can
obtain the next theorem.

\begin{theorem}\label{th4.20}
Let $0<H<\frac12$. If $f\in {\mathscr H}_{\ast}$, then the
generalized quadratic covariation $[f(B^H,\cdot),B^H]^{(W)}$ exists
and
\begin{align}
E\left|[f(B^H,\cdot),B^H]_t^{(W)}\right|^2\leq C_H\|f\|_{{\mathscr
H}_{\ast}}^2
\end{align}
for all $t\in [0,T]$.
\end{theorem}
By using the above result, we immediately get an extension of
It\^{o} formula stated as follows.
\begin{theorem}\label{th5.4}
Let $0<H<\frac12$ and let $F\in C^{1,1}({\mathbb R}\times {\mathbb
R}_{+})$. Suppose that the function $\frac{\partial }{\partial
x}F=f\in {\mathscr H}_{\ast}$. Then the It\^o type formula
\begin{align*}
F(B^H_t,t)=F(0,0)+&\int_0^tf(B^H_s,s)dB^H_s+\int_0^t\frac{\partial
}{\partial
t}F(B^H_s,s)ds+\frac12\left[f(B^H,\cdot),B^H\right]^{(W)}_t
\end{align*}
holds.
\end{theorem}
\begin{proof}
Similar to the proof of Theorem~\ref{th4.2}, we can use smoothing
procedure to prove our result. The main different key point is the
following approximation:
$$
F_n(x,s):=\int\int_{{\mathbb
R}^2}F(x-y,s-r)\zeta_n(y)\zeta_n(r)dydr,\qquad n\geqslant 1,
$$
where $\zeta_n$, $n\geq 1$ are the mollifiers defined
by~\eqref{sec4-eq00-4}.
\end{proof}
We next consider the integral
\begin{equation}\label{sec6-eq6.2}
\int_0^t\int_{\mathbb R}f(x,s){\mathscr L}^{H}(dx,ds),
\end{equation}
where $f$ is a deterministic function. For elementary function
$f_\triangle\in {\mathscr E}_{\ast}$ of the
form~\eqref{sec6-eq6.4-0} we define integration with respect to
local time ${\mathscr L}^{H}$ as follows
\begin{align*}
\int_0^t\int_{\mathbb R}f_\triangle(x,s)&{\mathscr
L}^{H}(dx,ds):=\sum_{i,j}f_{ij}\left[{\mathscr
L}^{H}(x_{i},s_{j})\right.\\
&\hspace{1cm}\left.-{\mathscr L}^{H}(x_{i},s_{j-1})-{\mathscr
L}^{H}(x_{i-1},s_{j})+{\mathscr L}^{H}(x_{i-1},s_{j-1})\right],
\end{align*}
for all $t\in [0,T]$. Notice that
\begin{align*}
{\mathscr L}^{H}&(x_{i},s_{j})-{\mathscr
L}^{H}(x_{i},s_{j-1})-{\mathscr L}^{H}(x_{i-1},s_{j})+{\mathscr
L}^{H}(x_{i-1},s_{j-1})\\
&=\left[{\mathscr L}^{H}(x_{i},s_{j})-{\mathscr
L}^{H}(x_{i-1},s_{j})\right]-\left[{\mathscr
L}^{H}(x_{i},s_{j-1})-{\mathscr L}^{H}(x_{i-1},s_{j-1})\right]\\
&=-\left[{1}_{(x_{i-1},x_{i}]}(B^H),B^H
\right]^{(W)}_{s_{j}}+\left[{1}_{(x_{i-1},x_{i}]}(B^H),B^H
\right]^{(W)}_{s_{j-1}}\\
&=-\left[{1}_{(x_{i-1},x_{i}]}(B^H) {1}_{(s_{j-1},s_{j}]}(\cdot),B^H
\right]^{(W)}_t
\end{align*}
for all $i,j$. We get the identity
\begin{equation}
\int_0^t\int_{\mathbb R}f_{\Delta}(x,s)\mathscr{L}^{H}(dx,ds)=
-\left[f_\triangle(B^H,\cdot),B^H\right]^{(W)}_t
\end{equation}
for all $t\in [0,T]$. Moreover, for $f\in {\mathscr H}_{\ast}$ we
can define
\begin{equation*}
\int_0^t\int_{\mathbb R}f(x,s){\mathscr L}^{H}(dx,ds):=\lim_{n\to
\infty}\int_0^t\int_{\mathbb R}f_{\triangle,n}(x,s){\mathscr
L}^{H}(dx,ds),\qquad{\text {in $L^2$}}
\end{equation*}
for all $t\in [0,1]$ if $f_{\triangle,n}\to f$ in ${\mathscr
H}_{\ast}$, as $n$ tends to infinity, where
$\{f_{\triangle,n}\}\subset {\mathscr E}_{\ast}$.
\begin{theorem}\label{prop6.1}
Let $0<H<\frac12$ and $f\in {\mathscr H}_{\ast}$. Then the
integral~\eqref{sec6-eq6.2} exists in $L^2$ and the Bouleau-Yor
identity takes the form
\begin{equation}\label{sec4-eq4.8}
\left[f(B^{H},\cdot),B^{H}\right]^{(W)}_t=-\int_0^t\int_{\mathbb
R}f(x,s){\mathscr L}^{H}(dx,ds)
\end{equation}
for all $t\in [0,T]$.
\end{theorem}
\begin{corollary}\label{cor6.2}
Let $0<H<\frac12$, $F\in C^{1,1}({\mathbb R}\times {\mathbb R}_{+})$
and $\frac{\partial}{\partial x}F=f\in {\mathscr H}_{\ast}$. Then
the It\^o type formula
\begin{align*}
F(B_t^{H},t)=F(0,0)+&\int_0^tf(B_s^{H},s)dB_s^{H}\\
&\qquad+\int_0^t\frac{\partial }{\partial t}F(B_s^{H},s)ds
-\frac12\int_0^t\int_{\mathbb R}f(x,s){\mathscr L}^{H}(dx,ds)
\end{align*}
holds.
\end{corollary}

Finally, let us consider the weighted local time of fBm $B^H$ with
$0<H<\frac12$ on a continuous curve. Let $a(t)$ denote a continuous
function on $[0,T]$. Then the function
$$
f_a(x,s)=1_{(-\infty,a(s))}(x)
$$
belongs to ${{\mathscr H}_{\ast}}$, and the integral
$$
\int_0^t\int_{\mathbb R}f_a(x,s){\mathscr L}^{H}(dx,ds)
$$
and the generalized quadratic covariation
$\left[f_a(B^{H},\cdot),B^{H}\right]^{(W)}$ exist in $L^2$. By the
idea due to Eisenbaum~\cite{Eisen1} and F\"ollmer {\it et
al}~\cite{Follmer}, as an example, we can show that the process
$$
\int_0^t\int_{\mathbb R}f_a(x,s){\mathscr L}^H(dx,ds),\quad t\geq 0
$$
is increasing and continuous. Thus, we can define the weighted local
time of $B^H$ with $0<H<\frac12$ at a continuous curve $t\mapsto
a(t)$ by setting
$$
{\mathscr L}^H(a(\cdot),t))=\int_0^t\int_{\mathbb
R}f_a(x,s){\mathscr L}^{H}(dx,ds).
$$



\end{document}